\documentclass{proc-l}
\PassOptionsToPackage{square,numbers}{natbib}
 \usepackage{natbib}

\usepackage[english]{babel}

\usepackage[colorlinks,linkcolor=red,anchorcolor=blue,citecolor=blue]{hyperref}

\usepackage{amssymb}

\usepackage{graphicx}

\usepackage[cmtip,all]{xy}

\usepackage{tikz}

\newtheorem{theorem}{Theorem}[section]
\newtheorem{lemma}[theorem]{Lemma}
\newtheorem{cor}[theorem]{Corollary}
\newtheorem{prop}[theorem]{Proposition}

\theoremstyle{definition}
\newtheorem{definition}[theorem]{Definition}

\newtheorem{remark}[theorem]{Remark}

\numberwithin{equation}{section}

\begin{document}

\title[Removable Singularities for constant curvature K\"{a}hler Metrics]{Locally Removable Singularities for K\"{a}hler Metrics with Constant Holomorphic Sectional Curvature}
%\title{Locally Removable Singularities for K\"{a}hler Metrics with Constant Holomorphic Sectional Curvature}

%    Only \author and \address are required; other information is
%    optional.  Remove any unused author tags.

%    author one information
% \author[short version for running head]{name for top of paper}
\author{Si-en Gong}
\address{Wu Wen-Tsun Key Laboratory of Math, USTC, Chinese Academy of Sciences. \newline \indent School of Mathematical Sciences, University of Science and Technology of China,\newline \indent Hefei, 230026, China}
\curraddr{Department of Mathematics, University of Kansas, \newline \indent Lawrence, KS 66045, USA}
\email{gse@mail.ustc.edu.cn}
\thanks{}

%    author two information
\author{Hongyi Liu}
\address{Department of Mathematics, University of California, Berkeley,\newline \indent Berkeley,  CA 94720, USA}
\curraddr{}
\email{hongyi@berkeley.edu}
\thanks{}

%   author three information
\author{Bin Xu}
\address{Wu Wen-Tsun Key Laboratory of Math, USTC, Chinese Academy of Sciences. \newline \indent School of Mathematical Sciences, University of Science and Technology of China,\newline \indent Hefei, 230026, China}
\email{bxu@ustc.edu.cn}

%    \subjclass is required.

\date{}

\dedicatory{}

%    "Communicated by" -- provide editor's name; required.
\commby{}

%    Abstract is required.
\begin{abstract}
Let $n\ge 2$ be an integer, and $B^{n}\subset \mathbb{C}^{n}$ the unit ball. Let $K\subset B^{n}$ be a compact subset such that $B^n\setminus K$ is connected, or $K=\{z=(z_1,\cdots, z_n)|z_1=z_2=0\}\subset \mathbb{C}^{n}$. By the theory of developing maps, we prove that a K\"{a}hler metric on $B^{n}\setminus K$ with constant holomorphic sectional curvature uniquely extends to $B^{n}$.
\end{abstract}

\maketitle

%    Text of article.

\section{Introduction}\label{intro}
Recently, there is a growing interest in the research of metrics with singularities. Donaldson\cite{Donaldson}, Li and Sun\cite{LiSun} investigated K\"{a}hler-Einstein metrics with cone singularities along a divisor on a K\"{a}hler manifold. There is also a lot of research on metrics with singularities on Riemann surfaces. For example, Nitsch\cite{nitsche1957isolierten}, Heins\cite{Hein} and Yamada\cite{yamada1988bounded} proved that an isolated singularity of a hyperbolic metric is either a conical singularity or a cusp one, and Heins\cite{Hein}, Mcowen\cite{mcowen1988point} and Troyanov\cite{Troyanov} independently gave a necessary and sufficient condition for the existence of a unique hyperbolic metric, which has the prescribed conical or cusp singularities, on a compact Riemann surface. Among all the research on singular metrics on Riemann surfaces, developing maps, due to \cite{Bryant,umehara2000metrics,eremenko2004metrics}, prove to be a very useful tool. By considering the monodromy of developing maps in \cite{CWWX}, Chen and coauthors constructed a new class of cone spherical metrics. Later, in \cite{FengXu}, Feng, Shi and Xu gave the explicit models of hyperbolic metrics near isolated singularities through developing maps, based on the related results mentioned above.

In this paper, we try to generalize the theory of developing maps to the case when dimension is larger than one. We hope that developing maps will still be a powerful tool in higher dimension. We aim to investigate metrics of constant holomorphic sectional curvature with cone singularities along a divisor. However, it is natural to ask first what will happen if the singularities are along a subvariety whose codimension is not one. We found that in this case the singularities are actually removable locally. Without loss of generality, we can consider the unit ball $B^{n}$ in $\mathbb{C}^{n}$ for convenience. In the hyperbolic and flat cases, we used the technique  similar to that in \cite{seshadri2006class}, which proves a simply connected complete K\"ahler manifold whose holomorphic sectional curvature is constant negative outside a compact set is biholomorphic to the unit ball. In the elliptic case, we applied an extension theorem in \cite{ivashkovich1984extension} to the developing map. In conclusion, by showing the existence of the developing map and extending it to the whole manifold, we get:
\begin{theorem}\label{Kcpt}
Let $n\ge 2$ be an integer, $B^{n}:=\{(z_1,\cdots, z_n)\in \mathbb{C}^{n}| \sum_{i=1}^{n}|z_i|^2<1\}$. Let  $K\subset B^{n}$ be a compact subset such that $B^n\setminus K$ is connected. $M=B^{n}\setminus K$ is endowed with a K\"{a}hler metric $g$ with constant holomorphic sectional curvature. Then $g$ uniquely extends to a K\"{a}hler metric with constant holomorphic sectional curvature on $B^{n}$.
\end{theorem}

\begin{theorem}\label{z1z2=0}
Let $n\ge 2$ be an integer, $B^{n}:=\{(z_1,\cdots, z_n)\in \mathbb{C}^{n}| \sum_{i=1}^{n}|z_i|^2<1\}$. $M:=B^{n}\setminus \{(z_1,\cdots,z_n)|z_1=z_2=0\}$ is endowed with a K\"{a}hler metric, $g$, with constant holomorphic sectional curvature. Then $g$ uniquely extends to a K\"{a}hler metric with constant holomorphic sectional curvature on $B^{n}$.
\end{theorem}

\begin{cor}\label{codim2}
Let $N$ be an analytic subvariety of codimension larger than $1$. Then a K\"{a}hler metric with constant holomorphic sectional curvature on $M\setminus N$ extends to $M$. In particular, if we let $N$ be a single point, then $N$ is a removable singularity of the K\"{a}hler metric.
\end{cor}

\begin{remark}\label{Riemisosing}
The similar statements do not hold in the real case. It turns out that a Riemannian metric $g$ of constant sectional curvature on ${\mathfrak{B}}^n\setminus \{0\}(n>2)$ can have an isolated singularity, where ${\mathfrak{B}}^n:=\{x\in \mathbf{R}^n: |x|<1\}$ is a real unit ball here. We will discuss it in detail in Section \ref{pro}.
\end{remark}

\begin{remark}\label{KEnonRemovable}
We want to mention an example of {\it a K\"{a}hler-Einstein metric on $B^2\setminus \{0\}$ that is singular at $0$ and does not extend to $B^2$}. To construct such an example, we begin with the Del Pezzo surface $X$  which is obatained by blowing up three points, $p_1, p_2, p_3\in \mathbf{C}P^2$, in general position. We denote by $E$ the exception divisor in $X$. Then $X\backslash E$ is isomorphic to $\mathbf{C}P^{2}\setminus \{p_1,p_2,p_3\}$ as quasi-projective varieties. Yum-Tong Siu \cite{siu1988existence} proves that there exists a smooth K\"{a}hler-Einstein metric $g$ on $X$.  We could look at the restriction of $g$ to $X\backslash E$ as a K\"{a}hler-Einstein metric  on $\mathbf{C}P^{2}\setminus \{p_1,p_2,p_3\}$, denoted by $g_0$. Finally we claim that $g_0$ does not extend to $\mathbf{C}P^2$, i.e. there exists $p_j$ with $1\leq j\leq 3$ such that
the restriction of $g_0$ to a punctured ball $B^2\backslash\{p_j\}$ centered at $p_j$ does not extend to $B^2$. Otherwise, by the uniqueness of the K\"{a}hler-Einstein metric on $\mathbf{C}P^2$, $g_0$ must be the Fubini-Study metric up to a scaling, hence have constant holomorphic sectional curvature. It follows that $g_0$ is of constant holomorphic sectional curvature.
By the smoothness of the original metric $g$ on $X$, $g$ also has constant holomorphic sectional curvature. By the compactness and simply connected property of $X$, $X$ should be the space form,  which is a contradiction.
\end{remark}

We explain the organization of this paper. In Section \ref{preliminaries}, we introduce some preliminaries. Most of the discussion in  Section \ref{preliminaries} is devoted to the existence of the developing map on a K\"{a}hler manifold with constant holomorphic sectional curvature. In the rest of Section \ref{preliminaries}, we state the classical Hartogs extension theorem and Weierstrass preparation theorem, which are used in the proof of the main result. In Section \ref{mainres}, using the developing map, we prove Theorem \ref{Kcpt} and \ref{z1z2=0}. Finally, in Section \ref{pro}, we discuss the possible further research on this topic.

\section{Preliminaries}\label{preliminaries}

\subsection{Cartan-Ambrose-Hicks Theorem}\label{CAHThm}

Cartan-Ambrose-Hicks Theorem is a classical theorem in Riemannian Geometry (\cite{Cheeger}). In this section, we recall its generalization on K\"{a}hler manifolds (\cite{FYZ}). This will be the start point of the developing map. For the convenience of readers, we provide some details for the proof of the theorems.

Let $(M,g)$ and $(M^{'},g^{'})$ be two $n$-dimensional Riemannian manifolds . Fix points $p\in M$, $p^{'}\in M^{'}$. Let $I:T_{p}M\to T_{p^{'}}M^{'}$ be an isometry. Take $r>0$ sufficiently small such that $B_{r}(p)$ and $B_{r}(p^{'})$ can be endowed with geodesic normal coordinates. Define the diffeomorphism
$\varphi=\exp_{p^{'}}\circ I\circ \exp_{p}^{-1}:B_{r}(p)\to B_{r}(p^{'})$.
Let $P_{\gamma}$ denote the parallel transport along a curve $\gamma$ on a Riemannian manifold, and $R$, $R^{'}$ the curvature tensors of $M$ and $M^{'}$, respectively.

For a geodesic $\gamma:[0,1] \to B_{r}(p)$ with $\gamma(0)=p$, write $\gamma^{'}:=\varphi(\gamma)$, and set $I_{\gamma}:=P_{\gamma^{'}}\circ I\circ P_{-\gamma}$. The following figure illustrates the definitions of $I_{\gamma}$ and $\varphi$:

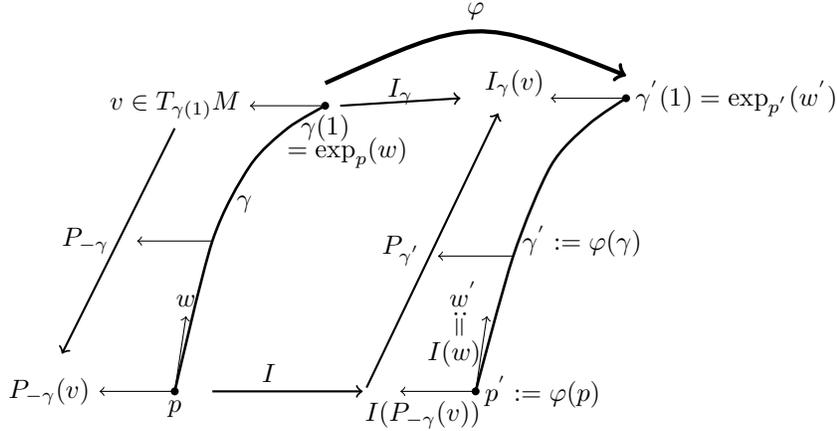
\begin{tikzpicture}
\draw [line width=1pt] plot[smooth] coordinates{(2,-2) (2.5,0) (3.0, 1) (3.5, 1.5) (4,1.8)};
\draw [line width=1pt] plot[smooth] coordinates{(6,-2) (6.5,-0.2) (7,1) (7.5, 1.55) (8,1.9)};

\fill (2,-2) circle (1.5pt);
\node[below] at (2,-2) {$p$};
\node[right] at (2.7,0.5) {$\gamma$};
\fill (6,-2) circle (1.5pt);
\node[right] at (6,-2) {$p^{'}:=\varphi(p)$};
\node[right] at (6.5,0) {$\gamma^{'}:=\varphi(\gamma)$};
\fill (4,1.8) circle (1.5pt);
\node[below] at (4,1.8) {$\gamma(1)$};
\node[below] at (4.3,1.5) {$=\exp_{p}(w)$};
\fill (8,1.9) circle (1.5pt);
\node[right] at (8,1.9) {$\gamma^{'}(1)=\exp_{p^{'}}(w^{'})$};
\draw [->] (4,1.8)--(3,1.8) node[left]{$v\in T_{\gamma(1)}M$};
\draw [->](2.5,0)--(1.5,0);
\draw [->](2,-2)--(1,-2) node[left]{$P_{-\gamma}(v)$};
\draw [->,thick] (2,1.5)--(0.5,-1.5);
\node[left] at (1.25,0){$P_{-\gamma}$};
\draw [->,thick] (2.5,-2)--(4.5,-2);
\node[above] at (3.25,-2) {$I$};
\draw [->] (6,-2)--(5,-2);
\node[below] at (5.3,-2) {$I(P_{-\gamma}(v))$};
\draw [->] (6.5,-0.2)--(5.5,-0.2);
\draw [->] (8,1.9)--(7,1.9);
\draw [->,thick] (4.55,-1.95)--(6.3,1.7);
\node[left] at (5.4,-0.125) {$P_{\gamma^{'}}$};
\draw [->,thick] (4.2,1.8)--(5.8,1.9);
\node [above] at (5, 1.7) {$I_{\gamma}$};
\node [right] at (6,2.1) {$I_{\gamma}(v)$};
\draw [->] (2,-2)--(2.15,-1) node[above]{$w$};
\draw [->] (6,-2)--(6.15,-1);
\node [left] at (6.15,-0.7) {$w^{'}$};
\node [left] at (6.0,-1.1) {\rotatebox{90}{=:}};
\node [left] at (6.2,-1.5) {$I(w)$};
\draw[ultra thick,->] (4,2.1).. controls (6,3) ..(8,2.2);
\node[below] at (6,3.3) {$\varphi$};
\end{tikzpicture}
\\
Let $\nabla$ be the Levi-Civita connection on a Riemannian manifold $(M,g)$. Recall that the curvature tensor $R(\cdot,\cdot)\cdot$ on $M$ is defined by
$$R(X,Y)Z:=-\nabla_X\nabla_YZ+\nabla_Y\nabla_XZ+\nabla_{[X,Y]}Z$$ for arbitrary $p$ on $M$ and $X,Y$ and $Z\in T_{p}M$. Then we have:

\begin{lemma}{(\citep[p.37, Lemma 1.35]{Cheeger})}\label{LocalCAH}
Let $(M,g)$ and $(M^{'},g^{'})$ be two Riemannian manifolds. Use the notations defined above. If for all geodesics $\gamma: [0,1]\to B_{r}(p)$ with $\gamma(0)=p$ and any $X, Y, Z\in T_{\gamma(1)}M$, there holds
$$I_{\gamma}(R(X,Y)Z)=R^{'}(I_{\gamma}(X),I_{\gamma}(Y))I_{\gamma}(Z),$$
then $\varphi$ is an isometry and $d\varphi_{\gamma(1)}= I_{\gamma}$.
\end{lemma}

This local property has its K\"{a}hler version. Before stating the proposition, we shall recall the concept of holomorphic sectional curvature. It is well known that the Riemann curvature tensor $Rm(\cdot,\cdot,\cdot,\cdot)$ is defined by
$$Rm(X,Y,Z,W):=g(R(X,Y)Z,W)$$
for arbitrary $p$ on $M$ and $X,Y,Z$ and $W\in T_{p}M$. On a K\"{a}hler mainfold $(M, J, g)$, let $\Pi_p$ be a plane of $T_{p}M$ that is invariant under $J$. It is easy to see that
$$K(\Pi_p):= Rm(X,JX,X,JX)$$
where $X\in \Pi_p$ and $|X|=1$, is independent of the choice of $X$ (\citep[Chap. IX.7]{Kobayashi}). We call $K(\Pi_p)$ the holomorphic sectional curvature by $\Pi_p$.

We will assume $M$ and $M^{'}$ are of constant holomorphic sectional curvature in the following, and set

\begin{eqnarray}
R_0(X,Y,Z,W):&=&\frac{1}{4}\bigg{(}g(X,Z)g(Y,W)-g(X,W)g(Y,Z)\nonumber\\
             & & +g(X,JZ)g(Y,JW)-g(X,JW)g(Y,JZ)\nonumber\\
             & &+2g(X,JY)g(Z,JW)\bigg{)}\nonumber
\end{eqnarray}
Then, we have

\begin{prop}(\citep[p.167 Prop. 7.3]{Kobayashi})\label{RcR0}
Let $(M,J,g)$ be a K\"{a}hler manifold of constant holomorphic sectional curvature $c$. Then $Rm=cR_0$.
\end{prop}

Then we are ready to give the proposition. It is well known to experts, but we will provide a proof of it since we could not find the suitable literature.

\begin{prop}\label{LocalKahlerCaH}
Let $(M,J,g)$ and $(M^{'},J^{'},g^{'})$ be two K\"{a}hler manifolds of constant holomorphic sectional curvature $c$. Then for each $p$ on $M$, each $q$ on $M^{ '}$, and each $I:T_pM\to T_qM^{'}$ with $I\circ J=J^{'}\circ I$, there exist neighborhoods $U \ni p$, $V\ni q$ and a holomorphic isometry $\varphi: U\to V$ such that $\varphi(p)=q$ and $d\varphi_{p}=I$.
\end{prop}

\begin{proof}
As above, construct the diffeomorphism $\varphi:=\exp_{p^{'}}\circ I\circ \exp_{p}^{-1}:B_{r}(p)\to B_{r}(q)$. We verify that $\varphi$ satisfies the condition of Lemma \ref{LocalCAH}. Indeed, by Proposition \ref{RcR0}, for each geodesic $\gamma: [0,1]\to B_{r}(p)$ with $\gamma(0)=p$ and $X, Y, Z, W\in T_{\gamma(1)}M$, since $I_{\gamma}=P_{\gamma^{'}}\circ I\circ P_{-\gamma}$, we have

\begin{align}
  &g^{'}\Big(I_{\gamma}\big(R(X,Y)Z\big),I_{\gamma}(W)\Big)\nonumber\\
  =& g^{'}\Big(I\circ P_{-\gamma}\big(R(X,Y)Z\big),I\circ
  P_{-\gamma}(W)\Big)\,(P_{\gamma}\; \text{is a parallel transport.})    \nonumber   \\
  =& g\Big(P_{-\gamma}\big(R(X,Y)Z\big),P_{-\gamma}(W)\Big)\;(I\, \text{is an isometry}) \nonumber \\
  =& g\big(R(X,Y)Z,W\big)\;(P_{-\gamma^{'}}\, \text{is a parallel transport})\nonumber \\
  =& Rm(X,Y,Z,W) \nonumber \\
  =& cR_{0}(X,Y,Z,W)\;(\text{Prop}. \ref{RcR0})\label{1}
\end{align}
By Proposition \ref{RcR0} again, we have
\begin{align}
 &g^{'}\Big(R^{'}\big(I_{\gamma}(X),I_{\gamma}(Y)\big)I_{\gamma}(Z),I_{\gamma}(W)\Big)\nonumber\\
=&Rm^{'}\big(I_{\gamma}(X),I_{\gamma}(Y),I_{\gamma}(Z),I_{\gamma}(W)\big)\nonumber\\
=&cR_{0}^{'}(I_{\gamma}(X),I_{\gamma}(Y),I_{\gamma}(Z),I_{\gamma}(W))\label{2}
\end{align}

On the other hand, the K\"{a}hler condition is equivalent to $\nabla J=0$, which implies $P_{\gamma}\circ J=J\circ P_{\gamma}$. Since

\begin{eqnarray}
I_{\gamma}\circ J &=& P_{\gamma^{'}}\circ I\circ P_{-\gamma}\circ J \nonumber\\
&=& P_{\gamma^{'}}\circ I\circ J\circ P_{-\gamma}\nonumber\\
&=& P_{\gamma^{'}}\circ J^{'}\circ I\circ P_{-\gamma}\nonumber\\
&=& J^{'}\circ P_{\gamma^{'}}\circ I\circ P_{-\gamma}\nonumber\\
&=& J^{'}\circ I_{\gamma}\nonumber
\end{eqnarray}
we have
\begin{eqnarray}\label{3}
R_{0}(X,Y,Z,W)=R_{0}^{'}\Big{(}I_{\gamma}(X),I_{\gamma}(Y),I_{\gamma}(Z),I_{\gamma}(W)\Big{)}
\end{eqnarray}
by the definitions of $R_0$ and $R_{0}^{'}$ and the fact that $I_{\gamma}$ preserves both the metric and the almost complex structure. Combining (\ref{1}), (\ref{2}) and (\ref{3}), we obtain
$$I_{\gamma}\big(R(X,Y),Z\big)=R^{'}\big(I_{\gamma}(X),I_{\gamma}(Y)\big)\big(I_{\gamma}(Z)\big)$$
This completes the verification of the condition, so by Lemma \ref{LocalCAH}, $\varphi$ is an isometry, and $d\varphi=I_{\gamma}$. In addition, $\varphi$ is holomorphic since $d\varphi\circ J=I_{\gamma}\circ J= J^{'}\circ I_{\gamma}=J^{'}\circ d\varphi$.

\end{proof}

Using Lemma \ref{LocalCAH}, one can show the following global isometry theorem:

\begin{theorem}{(\citep[p.41 Theorem 1.37]{Cheeger})}\label{CAHcc}
Let $M$ and $M^{'}$ be complete simply connected Riemannian manifolds of constant sectional curvature $c$, dim$M=$dim$M^{'}=n$. Then given any $p\in M$, $p^{'}\in M^{'}$ and an isometry $I: T_{p}M\to T_{p^{'}}M^{'}$, there exists an isometry $\Phi:M\to M^{'}$ such that $\Phi(p)=p^{'}$ and $d\Phi_{p}=I$. %In addition, $\forall q\in M$, there is a neighborhood $q\in U$ such that $\Phi=\exp_{\Phi(q)}\circ d\Phi_{q}\circ \exp^{-1}_{q}$.
\end{theorem}

The above theorem also has a K\"{a}hler version:

\begin{theorem}(\citep[p.170 Theorem 7.9]{Kobayashi})\label{KahlerCAH}
Let $(M,J,g)$ and $(M^{'},J^{'},g^{'})$ be two simply-connected complete K\"{a}hler manifolds of constant holomorphic sectional curvature $c$. For each $ p\in M$, $q\in M^{'}$ and each isometry $I:T_pM\to T_qM^{'}$ preserving the almost complex structure, there exists a unique holomorphic isometry $\varphi:M\to M^{'}$ such that $d\varphi_p=I$. In particular, Let $N_c$ be a complete simply connected K\"{a}hler manifold of constant holomorphic sectional curvature $c$. Then the holomorphic isometry group of $N_c$ acts transitively on $N_c$.
\end{theorem}

\subsection{Developing maps}\label{Devmap}

In this subsection, we review quickly the general theory of developing map in \cite{Thurston}, by which we show the existence of a developing map on a K\"{a}hler manifold of constant holomorphic sectional curvature.

Let $X$ be a connected real analytic manifold, and $G$ a subgroup of the group consisting of real analytic diffeomorphisms of $X$, acting transitively on $X$.

\begin{definition}{(\citep[p.139]{Thurston})}\label{GXmfld}
\textnormal{A $(G,X)$-manifold $M$ is a manifold with the property that for each $p\in M$, there exists a neighborhood $U_i\ni p$ and a diffeomorphism $\phi_i:U_i\to \phi_{i}(U_i)\subset X$ such that any transition function
$$\gamma_{ij}:=\phi_{i}\circ \phi_{j}^{-1}:\phi_j(U_i\cap U_{j})\to \phi_{i}(U_i\cap U_{j})$$
locally agrees with an element of $G$, i.e. $\forall r\in \phi_{j}(U_i\cap U_j)$, there is a neighborhood $U_r\ni r$ and an element $g$ of $G$ such that $\phi|_{U_r}=g|_{U_r}$.}
\end{definition}

\begin{remark}\label{germ}
In the language of sheaves, the chart $(U_i,\phi_i)$ is called a germ. Fixing two germs $(U_i,\phi_i)$ and $(U_j,\phi_j)$ and a point $x_0\in U_i\cap U_j$, we can define an equivalent relation: $(U_i,\phi_i)\sim (U_j,\phi_j)$ if there exists a neighborhood $V\subset U_i\cap U_j$ of $x_0$ such that $\phi_i$ and $\phi_j$ agree on $V$.
\end{remark}

For arbitrary two germs $(U_i, \phi_i)$ and $(U_j, \phi_j)$ ($U_i\cap U_j\not= \emptyset$), consider the transition function:
$$\gamma_{ij}=\phi_{i}\circ \phi_{j}^{-1}:\phi_{j}(U_i\cap U_j)\to \phi_i(U_i\cap U_j).$$
This naturally induces a locally constant map $\widetilde{\gamma_{ij}}:\phi_{j}(U_i\cap U_j)\to G$. Then $\widetilde{\gamma_{ij}}\circ \phi_{j}$ defines a locally constant map from $U_i\cap U_j$ to $G$, which we denote by $\gamma_{ij}$ by abusing the notations. Without loss of generality, we assume that $U_i\cap U_j$ is connected. Then $\gamma_{ij}:U_i\cap U_j\to G$ is actually a constant map. Therefore, $\gamma_{ij}$ can be viewed as an element of $G$. One can easily find that  $\widetilde{\phi}_{j}:=\gamma_{ij}\circ \phi_j$ agrees with $\phi_i$ on $U_i\cap U_j$.

Now we discuss the analytic continuation. Fix a germ $(U_0,\phi_0)$ and a point $x_0\in U_0$. Let $\pi: \widetilde{M}\to M$ be the universal cover of $M$. It is well known that $\widetilde{M}$ can be viewed as the space of homotopy classes of paths in $M$ that start at $x_0$. For an element of $\widetilde{M}$, take a path $\alpha$ representing it. We can choose points
$$x_0=\alpha(t_0),x_1=\alpha(t_1),\cdots, x_n=\alpha(t_n)$$
where $t_0=0$ and $t_n=1$, such that around $x_i$ $(1\le i\le n)$ there exist germs $(U_i,\phi_i)$ $(1\le i\le n)$ satisfying $\alpha|_{[t_i,t_{i+1}]}\subset U_i$ for all $i$. The following picture illustrates the construction:\\

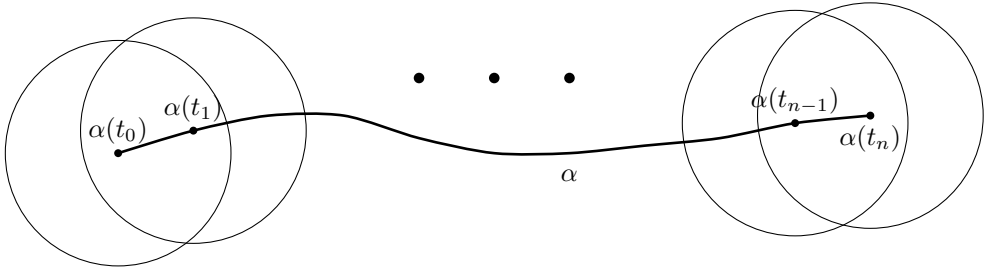
\begin{tikzpicture}
\draw [line width=1pt] plot[smooth] coordinates{(0,0) (1,0.3) (2,0.5) (3,0.5) (4,0.2) (5,0) (6,0) (7,0.1) (8,0.2)(9,0.4) (10,0.5)};
\fill (0,0) circle (1.5pt);
\draw (0,0) circle (1.5);
\node[above] at (0,0) {$\alpha(t_0)$};
\fill (1,0.3) circle (1.5pt);
\draw (1,0.3) circle (1.5);
\node[above] at (1,0.3) {$\alpha(t_1)$};
\fill (4,1) circle (2pt);
\fill (5,1) circle (2pt);
\fill (6,1) circle (2pt);
\draw (9,0.4) circle (1.5);
\fill (9,0.4) circle (1.5pt);
\node[above] at (9,0.4) {$\alpha(t_{n-1})$};
\draw (10,0.5) circle (1.5);
\fill (10,0.5) circle (1.5pt);
\node[below] at (10,0.5) {$\alpha(t_{n})$};
\node[below] at (6,-0.1) {$\alpha$};
 \end{tikzpicture}
\\
Again, without loss of generality, we can assume $U_i\cap U_{i+1}$ is connected.  By the previous discussion, for each $1<i\le n$, there exists an element $g_i$ in $G$ such that $g_i\circ \phi_i$, still denoted by $\phi_i$, agrees with $\phi_{i-1}$ on $U_i\cap U_{i-1}$. Going along $\alpha$, adjust $\phi_i$ in this way one by one. This process forms the analytic continuation of $\phi_0$ along $\alpha$. We use $(U_{n}^{\alpha}, \phi_{0}^{\alpha})$ to denote $(U_{n},\phi_{n})$. Then we have

\begin{prop}{(\citep[p.140]{Thurston})}\label{welldef}
The adjusted germ $(U_n,\phi_n)$ at $x_n=\alpha(1)$ depends only on the homotopy class that $\alpha$ belongs to. More precisely, if $\alpha$ and $\beta$ are in the same homotopy class, then $(U_{n}^{\alpha}, \phi_{0}^{\alpha})\sim (U_{n}^{\beta}, \phi_{0}^{\beta})$ (See Remark \ref{germ}).
\end{prop}

By Proposition \ref{welldef}, the developing map is well defined:

\begin{definition}{(\citep[p.140]{Thurston})}\label{devlopmap}
\textnormal{Given a point $x_0\in M$, consider a germ $(U_0,\phi_0)$ where $x_0\in U_0$. The \emph{developing map} of a $(G,X)$-manifold $M$ is the map $D:\widetilde{M}\to X$ such that
$$D=\phi_{0}^{\sigma}\circ \pi$$\
in some neighborhood of $\sigma$, for every $\sigma$ in $\widetilde{M}$.}
\end{definition}

Before going further, we will clarify two points about the developing map. The first is the uniqueness of the developing map. More precisely, we have
\begin{prop}\label{Uniqdevmap}
Let $D_1$ and $D_2$ be two developing maps on a $(G,X)$-manifold $M$. Then there exists a $g$ in $G$ such that $D_1=g\circ D_2$.
\end{prop}

\begin{proof}
Without loss of generality, fix $x_0\in M$ with two germs $(U_0,\phi_0)$ and $(U_0,\psi_0)$ such that $D_1|_{U_0}=\phi_0$ and $D_2|_{U_0}=\psi_0$ (Here $D_i|_{U_0}$ is actually a single-valued branch of $D_i$ on $U_0$). Indeed, one can take $(U,\phi_0)$ and $(V,\psi_0)$, and then take $U_0:=U\cap V$. Set $\widetilde{g}:=\phi_0\circ \psi_0^{-1}: \psi_0(U_0)\to \phi_0(U_0)$. By the definition of a $(G,X)$-manifold, there exists a $g\in G$ such that $g|_{U_0}=\widetilde{g}$. Now for any $x\in M$, consider a path $\gamma$ connecting $x_0$ and $x$ and cover it successively by $U_0,U_1,\cdots, U_m$ such that $D_1|_{U_i}=\phi_i$ and $D_2|_{U_i}=\psi_i$. Suppose $(U_i,\phi_i)$ and $(U_i,\psi_i)$ satisfy $\phi_i\circ \psi_i^{-1}=g$ on $\psi_i(U_i)$. Since $\phi_i^{-1}\circ \phi_{i+1}=id=\psi_{i}^{-1}\circ \varphi_{i+1}$ on $U_i\cap U_{i+1}$, $\phi_{i+1}\circ \psi_{i+1}^{-1}=\phi_i\circ \psi_i^{-1}$ on $\varphi(U_i\cap U_{i+1})$, which implies $\phi_{i+1}\circ \psi_{i+1}^{-1}=g$ on $\phi_{i+1}(U_{i+1})$. Therefore, by induction, we have $D_1(x)=g(D_2(x))$ for any $x$. Thus, $D_1=g\circ D_2$ on $M$.
\end{proof}
\noindent This will be used in the proof of the uniqueness of the extension of the metric.

The second is called the holonomy (or monodromy) of the developing map. It is possible that the holonomy contains the information of the geometry near singularities (see Section \ref{intro}). We begin with a point $x_0\in M$ an element $\alpha\in \pi_{1}(M,x_0)$. Consider the initial germ $\phi_0$. Then, analytic continuation along a loop representing $\alpha$ gives another germ $\phi_{0}^{\alpha}$ at $x_0$. By the proof of Proposition \ref{Uniqdevmap}, there exists a unique $g_{\alpha}\in G$ such that $\phi_{0}^{\alpha}=g_{\alpha}\circ \phi_0$. Thus, by Proposition \ref{welldef}, it gives a well-defined map $H: \pi_1(M,x_0)\to G$, assigning $\alpha\in \pi_1(M,x_0)$ to $g_{\alpha}\in G$. $H$ is called the holonomy of $M$. Note that For $\sigma,\tau\in \pi_1(M,x_0)$,
$$g_{\sigma}\circ g_{\tau}\circ \phi_{0}=g_{\sigma}\circ \phi_{0}^{\tau}\overset{(*)}{=}\widetilde{\phi_{0}^{\tau}}\circ \phi_0^{\sigma}=g_{\sigma\tau}\circ \phi_0$$
where $\widetilde{\phi_{0}^{\tau}}$ is an analytic continuation along $\tau$ with the initial germ $\phi_{0}^{\sigma}$ and $(*)$ is the result of the same argument as the proof in Proposition \ref{Uniqdevmap}. Therefore, $g_{\sigma\tau}=g_{\sigma}\circ g_{\tau}$, or equivalently, $H:\pi_{1}(M,x_0)\to G$ is a group homomorphism.

We now focus on the case when $M$ is a K\"{a}hler manifold of constant holomorphic sectional curvature $c$.  Now $N_c$ denotes the complete simply-connected K\"{a}hler manifold of constant holomorphic sectional curvature $c$, as stated in Theorem \ref{KahlerCAH}. Let $G$ be the holomorphic isometry group of $N_c$.

\begin{lemma}\label{GNc}
Let $(M,J,g)$ be a K\"{a}hler manifold of constant holomorphic sectional curvature $c$, then $M$ is a $(G,N_c)$-manifold.
\end{lemma}

\begin{proof}
 By Proposition \ref{LocalKahlerCaH}, for each $p\in M$, $q\in N_c$, there exist neighborhoods $U_p\ni p$ , $V_q\ni q$ and a holomorphic isometry $\phi_p:U_p \to V_q$ such that $\phi_p(p)=q$. It suffices to verify that $\{(U_p,\phi_p)\}_{p\in M}$ form a chart of a $(G,N_c)$-manifold. Indeed, if $U_{p_1}\cap U_{p_2}\not= \emptyset$, then $\psi_{p_1p_2}:=\phi_{p_1}\circ \phi_{p_2}^{-1}$ is a holomorphic isometry from $\phi_{p_2}(U_{p_1}\cap U_{p_2})$ to $\phi_{p_1}(U_{p_1}\cap U_{p_2})$ of $N_c$. Set $q_i:=\phi_{p_i}(p_i) (i=1,2)$.

We claim that $\psi_{p_1p_2}$ agrees locally with an element of $G$. In fact, for each $r_2\in V_{q_2}$, set $r_1:=\psi_{p_1p_2}(r_2)$. By Theorem \ref{KahlerCAH}, there exists a holomorphic isometry $\Phi$ defined on the whole $N_c$ such that $\Phi(r_2)=r_1$ and $(d\Phi)_{r_2}=(d\psi_{p_1p_2})_{r_2}$. Therefore, $\Phi=\psi_{p_1p_2}$ in a neighborhood of $r_2$.
\end{proof}

To sum up, we have:
\begin{theorem}\label{devexist}
Let $(M,J,g)$ be a (not necessarily complete) K\"{a}hler manifold of constant holomorphic sectional curvature $c$. Then there exists a holomorphic developing map $f:\widetilde{M}\to N_c$, where $\widetilde{M}$ is the universal cover of $M$. $f$ can also be viewed as a multivalued holomorphic local isometry from $M$ to $N_c$. The pull-back of the K\"{a}hler metric on $N_c$ by $f$ is exactly $g$.
\end{theorem}
\begin{proof}
By Lemma \ref{GNc}, $M$ is a $(G,N_c)$-manifold, so one can construct a developing map $f:\widetilde{M}\to N_c$ by Definition \ref{devlopmap}. To show $f$ is holomorphic, one only needs to note that the projection $\pi:\widetilde{M}\to M$ and the analytic continuation $\phi^{\sigma}_{0}$ are holomorphic. By assigning $x\in M$ to $f(\pi^{-1}(x))$, $f$ can be naturally viewed as a multivalued holomorphic map from $M$ to $N_c$. Finally,  since $\phi^{\sigma}_{0}$ is an isometry, $f$ is a local isometry, and the last assertion also follows immediately.
\end{proof}

\begin{remark}
When proving that the developing map $f$ is holomorphic, the essential part is that the analytic continuation $\phi^{\sigma}_{0}$ is holomorphic, which is actually showed when we prove that $\varphi$ is holomorphic in Proposition \ref{LocalKahlerCaH}. Different from \cite{CWWX}, in which they proved that the developing map is holomorphic by the fact that an orientation-preserving conformal map from a domain of $\mathbb{C}$ to $\mathbb{C}$ is holomorphic, the proof here only uses the K\"{a}hler condition and is more general in the sense that it does not rely on the dimension any more.
\end{remark}

Let $(M_c, g_c)$ be the complete and simply-connected Riemannian manifold of constant sectional curvature $c$. Let $G$ be the isometry group of $M_c$. Then $G$ acts transitively on $M_c$ by Theorem \ref{CAHcc}. The same argument can be applied to a Riemannian manifold $(M,g)$ of constant sectional curvature $c$ to show  the existence of the developing map on $M$ to the corresponding space form. More precisely,
\begin{theorem}\label{Riemdev}
Let $(M,g)$ be a Riemannian manifold of constant sectional curvature $c$. Then there exists a multivalued locally univalent real analytic map $f:M\to M_c$ such that $g=f^{*}g_c$.
\end{theorem}

\subsection{Functions of several Complex Variables}\label{Severalvar}

Before going to the proofs of the main results, we need to recall some results in complex analysis. One is Hartogs' extension theorem. We introduce two different versions here.

\begin{theorem}(\citep[p.30 Theorem 2.3.2]{Hormander})\label{HartogK}
Let $U$ be an open set of $\mathbb{C}^{n}$, $n>1$, and let $K$ be a compact subset of $U$ such that $U\setminus K$ is connected. Then for every holomorphic function $f$ on $U\setminus K$, $f$ extends holomorphically to $U$. The extension is unique.
\end{theorem}

\begin{theorem}(\citep[p.34 Theorem 1.25]{Voisin})\label{Hartogz1z2}
Let $U$ be an open set of $\mathbb{C}^{n}$, $n>1$, and $f$ a holomorphic function on $V:=U\setminus \{z|z_1=z_2=0\}$. Then there exists a unique holomorphic function $F$ on $U$ such that $F|_{V}=f$.
\end{theorem}

\begin{remark}\label{Kinducez1z2}
Actually, Theorem \ref{HartogK} induces Theorem \ref{Hartogz1z2}. Indeed, for a fixed $(0,0,w_3,\cdots, w_n)\in U$, consider $f|_{V}$, where $V=\{(z_1,z_2,w_3,\cdots, w_n)\in U||z_1|^2+|z_2|^2\not=0\}\cap U$. Then $f|_{V}$ can be viewed as a holomorphic map on $B^2_{r}\setminus\{0\}$($r>0$ is sufficiently small). Therefore, $f|_{V}$ can extend to a holomorphic map in two variables $z_1$ and $z_2$ on $\widetilde{V}=\{(z_1,z_2,w_3,\cdots, w_n)|z_1,z_2\in \mathbb{C}\}\cap U$ by Theorem \ref{HartogK}. In this way, $f$ holomorphically extends to $U$.
\end{remark}

Another is the Weierstrass preparation theorem, which gives the local geometry of the zero sets of holomorphic functions. First we give the definition of a Weierstrass polynomial.

\begin{definition}(\citep[p.7 Definition 1.1.5]{Daniel})\label{Wpoly}
\textnormal{A \emph{Weierstrass polynomial} is a polynomial in $z_1$  of the form
$$z_1^d+\alpha_{1}(w)z_1^{d-1}+\cdots+\alpha_{d}(w)$$
where $\alpha_i(w)$ are holomorphic functions on some small ball in $\mathbb{C}^{n-1}$ with $\alpha_i(0)=0$.}
\end{definition}

\begin{theorem}(\citep[p.8 Proposition 1.1.6]{Daniel})\label{WPT}
Let $f:P_{\epsilon}(0)\to \mathbb{C}$ be a holomorphic function on the polydisc $P_{\epsilon}(0):=\{(z_1,\cdots, z_n)| |z_i|<\epsilon_i, 1\le i\le n\}$, in which $\epsilon=(\epsilon_1,\cdots, \epsilon_n)$. Assume $f(0)=0$ and $f_{0}(z_1):=f(z_1,0,\cdots,0)\not\equiv 0$. Then there exists a unique Weierstrass polynomial $g(z_1,w)$ and a holomorphic function $h$ on some smaller polydisc $P_{\epsilon^{'}}(0)\subset P_{\epsilon}(0)$ such that $f=g\cdot h$ and $h(0)\not=0$.
\end{theorem}

\begin{cor}\label{noisozero}
A holomorphic function $f:U\to \mathbb{C}$ on an open set $U\subset \mathbb{C}^{n}$ does not have an isolated zero.
\end{cor}

\section{Proofs of the main results}\label{mainres}

We need some preparations first. The following lemma about the uniqueness of the extension is needed:

\begin{lemma}\label{uniextend}
Let $K\subset B^{n}=\{(z_1,\cdots, z_n) | \sum_{i=1}^{n}|z_i|^2 < 1\}$ such that $B^{n}\setminus K$ is a connected open subset of $B^{n}$. $B^{n}\setminus K$ is endowed with a K\"{a}hler metric $g$ of constant holomorphic sectional curvature $c$. If there exists an extended K\"{a}hler metric  $\widetilde{g}$ on $B^{n}$, whose holomorphic sectional curvature is $c$, then $\widetilde{g}$ is unique.
\end{lemma}

\begin{proof}
%By Proposition \ref{Uniqdevmap}, we first get that for any two developing maps $f_1, f_2: B^{n}\setminus K\to N_c$, where $N_c$ is the corresponding space form, $f_1=\tau\circ f_2$, where $\tau: N_c\to N_c$ is a holomorphic isometry.

For two metrics $\widetilde{g}_1$ and $\widetilde{g}_2$ that extend $g$ and whose curvatures are $c$, let $D_1, D_2:B^{n}\to N_c$ be the developing maps of $\widetilde{g}_1, \widetilde{g}_2$, respectively. Fix $x_0\in B^{n}\setminus K$ and take two germs $(U_0,\phi_0)$ and $(U_0,\psi_0)$, where $x_0\in U_0\subset B^{n}\setminus K$ , such that $D_1|_{U_0}=\phi_0$ and $D_2|_{U_0}=\psi_0$. Let $f_1$ and $f_2$ be the developing maps on $B^{n}\setminus K$ generated by $(U_0,\phi_0)$ and $(U_0,\psi_0)$, respectively. Then by Proposition \ref{Uniqdevmap}, we get $f_1=\tau\circ f_2$, where $\tau: N_c\to N_c$ is a holomorphic isometry.  In particular, $D_1|_{U_0}=\tau \circ D_2|_{U_0}$. This induces $D_1=\tau\circ D_2$. Therefore, $\widetilde{g}_1=\widetilde{g}_2$ since $\tau$ is a holomorphic isometry. This proves the lemma.

 %Then $D_1|_{B^{n}\setminus K}$ and $D_2|_{B^{n}\setminus K}$ are two developing maps on $B^{n}\setminus K$. Then by the previous discussion, $D_1|_{B^{n}\setminus K}=\tau\circ D_2|_{B^{n}\setminus K}$, and further $D_1=\tau\circ D_2$. Therefore, $\widetilde{g}_1=\widetilde{g}_2$ since $\tau$ is a holomorphic isometry. This proves the lemma.
\end{proof}

The next theorem makes the elliptic case trivial.

\begin{theorem}(\citep[Theorem 1]{ivashkovich1984extension})\label{extensionCPn}
Let $M$ be a Stein manifold, $U\subset M$ a domain in $M$, and $\widetilde{U}$ the envelope of holomorphy of $U$ (\cite[Section 5.4]{Hormander}). If $f:U\to \mathbf{C}P^{n}$ is a locally biholomorphic mapping, then there exists a locally biholomorphic mapping $F:\widetilde{U}\to \mathbf{C}P^{n}$ extending $f$.
\end{theorem}

To begin the proof of Theorem \ref{Kcpt}, we prove a useful lemma that will be used to prove the nondegeneration of the extended developing map in the hyperbolic and flat cases:

\begin{lemma}\label{zerononcpt}
Let $\Omega$ be a domain of $\mathbb{C}^{n}(n>1)$, and $f\not\equiv 0$ is a holomorphic function with $f(z_0)= 0$ for some $z_0\in \Omega$. Suppose that $Z:=\{z\in \Omega|f(z)=0\}$ satisfies $\Omega\setminus Z$ is connected. Then $Z$  is not a compact subset of $\Omega$.
\end{lemma}

\begin{proof}
Suppose otherwise $Z$ is compact.  Since $\Omega\setminus Z$ is connected, by Hartogs' Theorem (Theorem \ref{HartogK}), the holomorphic function $1/{f}:\Omega\setminus Z\to \mathbb{C}$ can be uniquely holomorphically extended to $\Omega$. Denote the extended function by $g$. This contradicts the fact that $f(z_0)=0$, because $h\equiv 1$ and $f\cdot g$ are both extensions of $f\cdot \frac{1}{f}$, but $f(z_0)\cdot g(z_0)=0\not=1 $, contradicting the uniqueness of the extension.
\end{proof}

We now turn to the proofs of the main theorems:

%\begin{proof}{(Theorem \ref{Kcpt})}
%By Lemma \ref{uniextend}, we can consider $B_r^{n}\subset B^{n}$ that contains $K$ instead of $K$. Thus we reduce the problem to Lemma \ref{BnminusBr}.
%\end{proof}

\begin{proof}{(Theorem \ref{Kcpt})}
Denote the ball of radius $r$ by $B_{r}^{n}:=\{(z_1,\cdots, z_n)\|\sum_{i=1}^{n}|z_i|^2 < r\}(n\ge 2)$. We can find an $r\in (0,1)$ such that  $K\subset B_r^{n}$ since $K$ is compact. Set $M:=B^{n}\setminus \overline{B_{r}^{n}}$. We first prove that $g|_{M}$ can extend to $B^n$.
 
 We claim that the K\" ahler manifold $M$ has a single-valued developing map $f_0: M \to N_c$. Actually, by Proposition \ref{LocalKahlerCaH}, for an arbitrary point $x_0$ in $M$, we could choose a germ $(U_0, \phi_0)$ for $(M,\,g|_M)$ such that  $x_0\in U_0\subset M$. Since $M$ is simple connected, doing analytical continuation of $(U_0,\phi_0)$ yields a single-valued developing map $f_0$ of $M$.

In the elliptic case, by applying Theorem \ref{extensionCPn} to $f_0$ one gets an extension $F:B^{n}\to \mathbf{C}P^{n}$. Since $F$ is locally biholomorphic, $\widetilde{g}:=F^{*}g_{FS}$ gives the extension of the metric we desire, where $g_{FS}$ is the Fubini-Study metric.

We now prove the hyperbolic case. Note that the same argument also works for the flat case without any difficulty. By Theorem \ref{HartogK}, we can extend $f_0$ to $F:B^{n}\to \mathbb{C}^{n}$. Consider the function $h:B^{n}\to \mathbb{C}$ defined by $h(z)=\sum_{i=1}^{n}|F_i|^{2}$. We see that
 $h$ is subharmonic by the following simple computation
 $$ \frac{\partial^2 |F_i|^2}{\partial z_j \partial \bar{z_j}}=\frac{\partial^2 \big(F_i\,\overline{F_i}\big)}{\partial z_j \partial \bar{z_j}}=\left|\frac{\partial F_i}{\partial z_j}\right|^2\geq 0.  $$
  Choose $r_0$ such that $r<r_0<1$. By the maximum principle, one can see for every $x\in \overline{B_r^{n}}$, $h(x)\le \sup_{y\in B_{r_0}^{n}}h(y)< 1$. This implies that the extended $F$ has its image in $\mathbb{B}^{n}$. Next, We claim that $dF$ is nondegenerate everywhere. In fact, denote the zero set of $\det dF$ by $Z$. Suppose $Z\not=\emptyset$. Since $Z\subset \overline{B_r^{n}}$ is closed  and bounded, $Z$ is compact. However, by Lemma \ref{zerononcpt}, $Z$ cannot be compact, a contradiction. Therefore, $Z=\emptyset$, i.e. $dF$ is non-degenerate everywhere on $B^{n}$. In conclusion, $\widetilde{g}= F^{*}g_{\mathbb{B}^{n}}$ defines a metric on $B^{n}$ that extends $g|_{M}$, with the curvature $-1$.

Moreover, the uniqueness of extension is guaranteed by Lemma \ref{uniextend}. Lemma \ref{uniextend} also shows that $g=\widetilde{g}$ on $B^{n}\setminus K$.

%Finally, we explain why $g=\widetilde{g}$ on $B^{n}\setminus K$. For any $x\in B^{n}\setminus K$, if  $x\in M$, then it is obvious. Otherwise, since $B^{n}\setminus K$ is connected, we can find a path $\gamma$ joining $x_0$ and $x$. Let $(U_x,\phi_x)$ be the germ generated by the analytic continuation along $\gamma$ beginning with $(U_0,\phi_0)$. Then $g=\phi_x^{*}g_{N_c}$. Note that $F|_{U_0}=\phi_0$. Since $(U_x,F|_{U_x})$ is generated by the  analytic continuation of $(U_0, F|_{U_0})$ along $\gamma$,  we have $F|_{U_x}=\phi_x$. Therefore, $\widetilde{g}_{x}=F^{*}g_{N_c}|_{F(x)}=\phi_x^{*}g_{N_c}|_{F(x)}=g_x$.

\end{proof}

%\begin{lemma}\label{BnminusBr}
%Let $B_{r}^{n}=\{(z_1,\cdots, z_n)\|\sum_{i=1}^{n}|z_i|^2 < r\}$, where $0<r<1$ and $n\ge 2$. $M:=B^{n}\setminus \overline{B_{r}^{n}}$ is endowed with a K\"{a}hler metric $g$ with constant holomorphic sectional curvature. Then $g$ uniquely extends to a K\"{a}hler metric with constant holomorphic sectional curvature on $B^{n}$ .
%\end{lemma}

Now we prove the other result. We begin with the following lemma:

\begin{lemma}\label{z1z2zero}
Let $h: B^{n}\setminus \{z|z_1=z_2=0\}\to \mathbb{C}$ be holomorphic. If $h$ nowhere vanishes, then so is its analytic continuation $H: B^{n}\to \mathbb{C}$.
\end{lemma}
\begin{proof}
For a fixed $w=(0,0,w_3,\cdots, w_n)\in B^{n}$, $h|_{V}$, where $V=\{(z_1,z_2,w_3,\cdots, w_n)\in B^{n}||z_1|^2+|z_2|^2\not=0\}$, can be viewed as a holomorphic map in two variables. In addition, $h$ nowhere vanishes on $V$ by the assumption.

Set $\widetilde{V}=\{(z_1,z_2,w_3,\cdots, w_n)|z_1,z_2\in \mathbb{C}\}\cap B^{n}$. $H|_{\widetilde{V}}$ is the extension of $h|_{V}$, which can also be viewed as a holomorphic map in two variables. Then by Corollary \ref{noisozero}, $H(w)\not= 0$, which completes the proof.
\end{proof}

Then we can prove Theorem \ref{z1z2=0}:

\begin{proof}{(Theorem \ref{z1z2=0})}
By the same argument as in the proof of Theorem \ref{Kcpt}, it is proved in the elliptic case.

We then prove the hyperbolic case. Consider the developing map $f:M\to \mathbb{B}^{n}$, where $\mathbb{B}^{n}$ is the Bergman ball with the metric $g_{\mathbb{B}^{n}}$. Since $M$ is simply connected, $f$ is a single-valued local isometry, so $df$ is nondegenerate at every point of $M$,i.e. $\det df\not =0$ at every point of $M$. In addition, $f$ is holomorphic, so by Hartogs' Theorem (Theorem \ref{Hartogz1z2}), it extends to a holomorphic map on $B^{n}$, which is still denoted by $f: B^{n}\to \mathbb{C}^{n}$. Note that im$f$ is contained in $\overline{B^{n}}$. Applying Lemma \ref{z1z2zero} to $\det df$, one can find that $df$ is nondegenerate everywhere in $B^{n}$. Since $df$ is nondegenerate everywhere, $f$ is an open map. Therefore $f:B^{n}\to \overline{B^{n}}$ has its image in $\mathbb{B}^{n}$.

To conclude, $\widetilde{g}= f^{*}g_{\mathbb{B}^{n}}$ defines a metric on $B^{n}$ that extends $g$, with curvature $-1$. In addition, the uniqueness is guaranteed by Lemma \ref{uniextend}.

The same argument also works for the flat case without any difficulty.
\end{proof}

\section{Some Prospects}\label{pro}

 We could get more results about removable singularities for K\"{a}hler metrics with constant holomorphic sectional curvature by using other versions of Hartogs' extension theorem. More precisely, let $U\subset \mathbb{C}^{n}(n>1)$ be a domain, $K\subset U$. $f:U\setminus K\to \mathbb{C}$ is  holomorphic. We have seen when $K$ is compact or $K=\{z_1=z_2=0\}$, $f$ extends to $U$. We wonder whether any other conditions can be imposed on $K$ such that $f$ can extend to $U$. Referring to the proofs of Theorems \ref{Kcpt} and \ref{z1z2=0}, it is hopeful to get other forms of removable singularities.

In addition, instead of a manifold of constant sectional curvature (resp. constant holomorphic sectional curvature), we also expect the existence of the developing map on a locally Riemannian (resp. Hermitian) symmetric space. This may give us the consequences similar to the case of constant sectional curvature(resp. constant holomorphic sectional curvature).

Moreover, recall that in \cite{FengXu}, an investigation into the monodromy of the developing map gives the explicit classified moduli of conformal hyperbolic
metrics near isolated singularities, and Theorem \ref{Kcpt} and \ref{z1z2=0} shows us the removability of singularities of codimension $0$ or larger than $1$, so in the next step, we naturally plan to use the theory of developing maps to study the codimension-one singularities of hyperbolic metrics in higher dimension:
\\
\\
\noindent{\bf Problem:} What is the asymptoic behavior of a K\"{a}hler metric with constant negative holomorphic sectional curvature in $B^{n}\setminus\{z_1z_2\cdots z_n=0\}$? For example, $-\sqrt{-1}\partial \overline{\partial}\Big(1-\sum_{j=1}^n\,|z_j|^{2\beta_j}\Big)$ with $0<\beta_j\not=1$
 is such a K\" ahler form (metric) in  $B^{n}\setminus\{z_1z_2\cdots z_n=0\}$  with cone singularities along the divisor $\{z_1z_2\cdots z_n=0\}$.
\\
\\
\indent Finally, we elaborate on the Remark \ref{Riemisosing}. Denote the real unit ball by $\mathfrak{B}^{n}$. For the given sectional curvature $c=-1,0$ or $1$, we want to construct the single-valued developing maps $f_{-1}$, $f_{0}$ and $f_{1}$ on $\mathfrak{B}^{n}\setminus \{0\}$ with image in $\mathfrak{B}^{n}$, ${\Bbb R}^{n}$ and $S^{n}$ (see Theorem \ref{Riemdev}) respectively such that their Jacobians are nondegenerate everywhere but cannot analytically extend to $\mathfrak{B}^{n}$. It suffices to give the expression of $f_{-1}$ solely since it induces the other two naturally.  We find $f_{-1}:\mathfrak{B}^{n}\setminus \{0\}\to \mathfrak{B}^{n}, x\mapsto \frac{x}{2}+\frac{||x||}{4}e_1$,where $e_1=(1,0,\cdots,0)$, meets the requirements. Therefore, if we set $(M_c, g_c)$ $(c=-1,0,1)$ as the space form, $f_{c}^{*}g_c$ defines a Riemannian metric on $B^{n}\setminus \{0\}$ that has an isolated singularity at $0$. Also, it is natural to ask whether we can find a way to classify the singularities.

\section*{Acknowledgments}
The authors would like to express his sincere gratitude to Professor Song Sun, Professor Qiongling Li and Dr. Martin de Borbon for their valuable comments and advice. The authors thank Dr. Jingchen Hu very much for his valuable comments and for showing them the reference \cite{ivashkovich1984extension}, which helps them solve the elliptic case of Theorems \ref{Kcpt} and \ref{z1z2=0}. Also, the authors thank Professor Pietro Majer for telling them on MathOverFlow the examples of non-degenerate real analytic maps that cannot extend to the origin (see \cite{PMajerhelp}).  B.X. is supported in part by the National Natural Science Foundation of China (Grant nos. 11571330 and 11971450) and the Fundamental Research Funds for the Central Universities. Part of the work was completed while B.X. was visiting Institute of Mathematical Sciences at ShanghaiTech University in Spring 2019.

%    Bibliographies can be prepared with BibTeX using amsplain,
%    amsalpha, or (for "historical" overviews) natbib style.
\bibliographystyle{amsalpha}
\bibliography{references2}
%    Insert the bibliography data here.

\end{document}